\documentclass[10pt]{amsart}
\usepackage[utf8]{inputenc}
\usepackage{latexsym}
\usepackage{amscd, amsfonts, eucal, mathrsfs, amsmath, amssymb, amsthm}
\input xy
\xyoption{all}

\newcommand{\field}[1]{\mathbf #1}
\newcommand{\mf}[1]{\mathfrak #1}
\newcommand{\mc}[1]{\mathcal #1}
\newcommand{\ms}[1]{\mathscr #1}
\newcommand{\widebar}[1]{\overline{#1}}

\newcommand{\R}{\field R}

\newcommand{\F}{\field F}
\newcommand{\Z}{\field Z}
\newcommand{\Q}{\field Q}

\newcommand{\simto}{\stackrel{\sim}{\to}}

\newcommand{\Spec}{\operatorname{Spec}}
\newcommand{\spec}{\operatorname{Spec}}
\newcommand{\Spf}{\operatorname{Spf}}

\renewcommand{\P}{\field P}

\DeclareMathOperator{\Pic}{Pic}

\DeclareMathOperator{\GL}{GL}

\renewcommand{\H}{\operatorname{H}}

\newcommand{\ch}{\operatorname{char}}

\renewcommand{\[}{[\!\hspace{0.03em}[}
\renewcommand{\]}{]\!\hspace{0.03em}]}

\DeclareMathOperator*{\tensor}{\otimes}

\DeclareMathOperator{\rk}{\operatorname{rk}}

\newcommand{\id}{\operatorname{id}}

\DeclareMathOperator{\Aut}{\operatorname{Aut}}

\DeclareMathOperator{\Hom}{\operatorname{Hom}}
\DeclareMathOperator{\Ext}{\operatorname{Ext}}

\DeclareMathOperator{\NS}{NS}

\newcommand{\eps}{\varepsilon}

\DeclareMathOperator{\Nef}{Nef}

\renewcommand{\sp}{\operatorname{sp}}

\DeclareMathOperator{\colim}{colim}

\newtheorem{lem}{Lemma}[section]

\DeclareMathOperator{\Def}{Def}

\newtheorem{thm}[lem]{Theorem}

\newtheorem{prop}[lem]{Proposition}

\newtheorem{cor}[lem]{Corollary}

\theoremstyle{definition}

\theoremstyle{remark}
\newtheorem{remark}[lem]{Remark}

\title{A note on the cone conjecture for K3 surfaces in positive characteristic}
\author{Max Lieblich}
\email{lieblich@uw.edu}
\author{Davesh Maulik}
\email{maulik@mit.edu}

\begin{document}
\maketitle
\begin{abstract}
We prove that, for a $K3$ surface in characteristic $p>2$, the automorphism group acts on the nef cone with a rational polyhedral fundamental domain and on the nodal classes with finitely many orbits.  As a consequence, for any non-negative integer $g$, there are only finitely many linear systems of irreducible curves on the surface of arithmetic genus $g$, up to the action of the automorphism group.
\end{abstract}
\tableofcontents

\section{Introduction}\label{S:intro}

Given any K3 surface over the complex numbers, it is a theorem of Sterk (based
on ideas of Looijenga) \cite{sterk} that the automorphism group acts on the nef cone of the
surface with a rational polyhedral fundamental domain. In this note, we extend
these results to characteristic $p>2$. In particular, any K3 surface in this
setting carries only finitely many elliptic pencils, up to the action of the
automorphism group, answering a question of A. Kumar to the second author.

We split the problem into three cases, according to the height and Picard number
of the surface. For K3 surfaces of finite height, results of Katsura and van der
Geer allow us to lift the surface along with its full Picard group to
characteristic $0$. We can then directly apply Sterk's theorem by a
specialization argument. For K3 surfaces with Picard number $22$, we will use
Ogus's supersingular Torelli theorem to reduce the statement to an analogue of
Sterk's original theorem. At this stage, thanks to recent proofs of the Tate
conjecture \cite{charles-tate, charles-tate-erratum, charles-tate-zarhin,
mp-tate, mpw-tate, maulik-tate}, we could be done. However, there is a simpler
deformation-theoretic argument for dealing with Artin-supersingular K3s of
Picard number less than $22$ that does not assume the Tate conjecture for
supersingular K3 surfaces. We include this argument here in Section
\ref{S:violation} in case it proves useful for other purposes.

These results are almost certainly known to experts in the field, but they
appear to have escaped appearing in the literature. In characteristic $0$,
Sterk's theorem is a special case of the cone conjecture of Morrison-Kawamata,
predicting such behavior for all smooth projective varieties with numerically
trivial canonical bundle. We refer the reader to \cite{totaro2} for a beautiful
survey of the history and known results. While the authors are optimistic that
several of Totaro's more general results \cite{totaro} should extend using the
methods here, higher-dimensional results remain mysterious even in
characteristic $0$.

A brief sketch of the contents of this note: In section \ref{S:special} we show
that for a family of K3s over a dvr, if the specialization map on Picard
groups is an isomorphism then it is equivariant with respect to the
specialization map on automorphism groups. In sections \ref{S:violation} and
\ref{S:lifting} we show how to create various families with constant Picard
group and improved generic fiber. In section \ref{S:Shioda} we prove the main
theorem for K3 surfaces with Picard number 22, and in section \ref{S:main} we
prove the result in general.

The authors would like to thank O.\ Benoist, H.\ Esnault, N.\ Fakhruddin, A.\ Kumar, C.\ 
Liedtke, G.\ Martin, C.\ McMullen, V.\
Srinivas, and the referee for helpful questions, comments, and corrections. The first author is
partially supported by NSF grant DMS-1021444, NSF CAREER grant DMS-1056129, and
the Sloan Foundation. The second author was partially supported by a Clay
Research Fellowship.

\section{Some results on specializations of automorphisms and Nef cones}\label{S:special}

Fix a Henselian dvr $R$ with algebraically closed residue field $k$ and fraction
field $K$. Fix an algebraic closure $\widebar K\supset K$. Given a finite
extension $K\subset L$, write $R_L$ for the integral closure of $R$ in $L$; this
is again a complete dvr.

\begin{thm}\label{specialize automorphisms} Let $\pi:X\to\spec R$ be a smooth
projective relative surface whose special fiber is not birationally ruled. Write

$$\sp:\Pic(X_{\widebar K})\to\Pic(X_k)$$ for the natural specialization map. If
$\sp(\ms L)$ is ample for every ample invertible sheaf $\ms L$ on
$X_{\widebar K}$ then there is a group
homomorphism $$\sigma:\Aut(X_{\widebar K})\to\Aut(X_k)$$ such that $\sp$ is
$\sigma$-equivariant with respect to the natural pullback actions. In addition,
if $\H^0(X, T_X)=0$ then $\sigma$ is injective.
\end{thm}
\begin{proof}
 Given a pair of field extensions $L'/L/K$, the faithful flatness of $L'/L$
implies that the pullback map $$\Aut(X_L)\to\Aut(X_{L'})$$ is
injective. On the other hand, given any algebraic field extension $L/K$, any
automorphism of $X_L$ is defined over some finite subextension. We thus see that for a fixed algebraic
closure $\widebar K/K$ the pullback maps define a canonical isomorphism
 \begin{equation}\label{Eq:colimit}
 \colim_{K\subset L\subset\widebar K}\Aut(X_L)\to\Aut(X_{\widebar
K}),\end{equation} with the colimit taken over all finite subextensions. For each
finite extension $L/K$, let $R_L$ denote the normalization of $R$ in $L$; since
$R$ is Henselian, $R_L$ is again a Henselian dvr with residue field $k$.
 
 Since $X_k$ is not birationally ruled, the Matsusaka-Mumford theorem (Corollary
1 of \cite{MM}) implies that for each $L$ the restriction map
 $$\Aut(X_{R_L})\to\Aut(X_L)$$
is an isomorphism. Restricting to $X_k$ thus yields a morphism
$$\Aut(X_L)\to\Aut(X_k).$$
When $\H^0(X_k,T_{X_k})=0$ this map is injective, as the (slightly more general)
Proposition \ref{P:injective spec} below shows.

For any extension $L'/L$, the diagram

\begin{equation}\label{Eq:aut diagram}
\xymatrix{\Aut(X_L)\ar[dd] & \Aut(X_{R_L})\ar[dr]\ar[dd]\ar[l] & \\
& & \Aut(X_k)\\
\Aut(X_{L'}) & \Aut(X_{R_L'})\ar[ur]\ar[l]}
\end{equation}
induced by base change commutes, yielding the desired group homomorphism . 
$$\sigma:\Aut(X_{\widebar K})\to\Aut(X_k).$$
Since the colimit in diagram \eqref{Eq:colimit}
has injective transition maps, we also see that $\sigma$
is injective when $\H^0(X_k, T_{X_k})=0$.

The diagram \eqref{Eq:aut diagram} acts in the natural way on the restriction
diagram
\begin{equation}\label{Eq:pic diagram}
\xymatrix{\Pic(X_L)\ar[dd] & \Pic(X_{R_L})\ar[dr]\ar[dd]\ar[l] & \\
& & \Pic(X_k)\\
\Pic(X_{L'}) & \Pic(X_{R_L'})\ar[ur]\ar[l]}
\end{equation}
of Picard groups.  Since the resulting map
$$\Pic(X_{\widebar K})\simto\colim\Pic(X_{L})\to\Pic(X_k)$$
defines the usual specialization map, we conclude that $\sp$ is
$\sigma$-equivariant, as claimed.
\end{proof}

\begin{prop}\label{P:injective spec}
Suppose $\H^0(X_k,T_{X_k})=0$. If $U\subset X$ is the complement of finitely many
$k$-valued points, then any
automorphism $f$ of $U$ that acts as the identity on $U_k$ equals the identity.
\end{prop}
\begin{proof}
 Consider the restriction of $f$ to the formal completion $\widehat U_{U_k}$.
Write $U^{(n)}$ for $U\tensor R/\mf m_R^{n+1}$. The lifts of $f|_{U^{(n)}}$ to
an automorphism of $U^{(n+1)}$ are a torsor under
$\H^0(U_k,T_{U_k})=\H^0(X_k,T_{X_k})=0$, which implies that $f|_{\widehat U}=\id$.
Writing $\xi\in U_k$ for the generic point, we conclude that the induced
automorphism $f_{\xi}$ of $\widehat{\ms O}_{U,\xi}$ equals the identity,
whereupon the restriction of $f$ to the generic point of $X_K$ fits into a
commutative diagram
 $$\xymatrix{K(X_K)\ar[r]^f\ar[d] & K(X_K)\ar[d]\\
 K(\widehat{\ms O}_{U,\xi})\ar[r]_f & K(\widehat{\ms O}_{U,\xi}),}$$ with each
vertical arrow the canonical inclusion into the completion. We conclude that
$f|_{K(U_K)}=\id$, whence $f=\id$ since $U$ is separated.
\end{proof}

In this paper, we will use Theorem~\ref{specialize automorphisms} for a
family of K3 surfaces $X\to\Spec R$. The special fiber is not birationally ruled,
since, for example, it is
simply connected and not rational. (Of course, such surfaces can be uniruled.)
When the specialization map for the Picard groups of a family of K3
surfaces is an isomorphism, there are a few consequences for the cone of curves
that we will also use in the sequel.

\begin{lem}\label{curves deform}
If $X\to\Spec R$ is a relative K3 surface such that the specialization map
$\sp$ on Picard groups is an isomorphism, then any integral curve
$C_k\subset X_k$ is the closed fiber of a flat family of closed subschemes
$C\subset X_{R_L}$ for some finite extension $L$.
\end{lem}
\begin{proof}
By the assumption on specialization, the invertible sheaf $\ms O_{X_k}(C)$ lifts
to an invertible sheaf $\ms L$ on some such $X_{R_L}$. Consider the cohomology
groups $\H^i(X_k,\ms L_k)$ for $i=1,2$. Since $\ms L_k$ is effective and
non-trivial, Serre duality and the fact that $X_k$ is K3 imply that
$\H^2(X_k,\ms L_k)=0$ and $\ms O_{C_k}(C_k)\cong\omega_{C_k}$. The sequence
$$0\to\ms O_{X_k}\to\ms O_{X_k}(C_k)\to\ms O_{C_k}(C_k)\to 0$$
gives rise to a sequence
$$0\to\H^1(X_k,\ms L_k)\to\H^1(C_k,\omega_{C_k})\to\H^2(X_k,\ms O_X)\to 0.$$
Grothendieck-Serre duality shows that each of the latter two vector spaces are
$1$-dimensional, whence we conclude that $\H^1(X_k,\ms L_k)=0$.

It follows from cohomology and base change that $(\pi_L)_\ast\ms L$ is a locally
free $R_L$-module whose sections functorially compute $\H^0(\ms L)$. In
particular, there is a section $s:\ms O_{X_L}\to\ms L$ lifting the section with
vanishing locus $C_k$. The vanishing locus of $s$ gives the desired closed
subscheme $C\subset X_L$.
\end{proof}

\begin{cor}\label{C:nef preserved}
If $X\to\Spec R$ is a relative K3 surface such that specialization map
$\sp$ is an isomorphism then $\sp$ 
induces an isomorphism of nef cones $\Nef(X_{\widebar K})\to\Nef(X_k)$
respecting the ample cones.
\end{cor}
\begin{proof}
 First observe that the specialization map preserves the effective
 cones. Indeed, by assumption the cones lie in canonically isomorphic
 spaces. Specialization inserts one cone into the other.
 By Lemma \ref{curves deform}, this inclusion is surjective. But the
 nef cone is determined by intersecting with elements of the 
effective cone, giving the result. Since this is an isomorphism that respects
the inner product and the ample cone is the interior of the nef cone, this
isomorphism respects the ample cones.
\end{proof}

\section{Deforming violations of Artin's conjecture to finite height}\label{S:violation}

This section is not needed for the argument if one accepts the full
Tate conjecture for K3 surfaces (which was proven after this was
originally written \cite{charles-tate,charles-tate-erratum,charles-tate-zarhin,mp-tate,mpw-tate,maulik-tate}). We include it
here because it is significantly simpler than the proof of the Tate
conjecture and the argument may be of independent interest.

Suppose $X$ is an Artin-supersingular surface. If the Picard rank of $X$ is
greater than 4 then $X$ is elliptic and therefore Shioda-supersingular by
Theorem 1.7 of \cite{artin}. (To see that any $X$ with Picard number at least
$5$ is elliptic, note that the Picard lattice represents $0$ by the
Hasse-Minkowski theorem and the Hodge index theorem, and once the lattice
represents $0$ it is a standard exercise that the surface contains an elliptic
curve -- e.g., exercise IX.6 of \cite{beauville}.) Thus, we will assume that $X$
has Picard rank $\rho \leq 4$.

\begin{prop}\label{P:finitize the height}
There is a smooth projective relative K3 surface $\mc X$ over $k\[t\]$ such
that
\begin{enumerate}
\item $\mc X_0$ is isomorphic to $X$
\item the restriction map $\Pic(\mc X)\to\Pic(X)$ is an isomorphism
\item the generic fiber $\mc X_{\eta}$ is of finite height
\end{enumerate}
\end{prop}
\begin{proof}
Fix generators $L_1,\ldots,L_\rho$ for $\Pic(X)$. By Proposition 1.5 of
\cite{deligne}, we know that the locus in $\Spf k\[t_1,\ldots,t_{20}\]$ over
which each $L_i$ deforms is determined by a single equation. We conclude that
there is a closed subscheme $\Spec A\subset\Spec k\[t_1,\ldots,t_{20}\]$ of
dimension at least 16 and a universal formal deformation of
$(X;L_1,\ldots,L_\rho)$ over $\Spf A$. Since some linear combination of the
$L_i$ is ample, this formal deformation is algebraizable. It suffices to show
that the generic fiber is not supersingular.

By Artin approximation, we can realize $A$ as the completion of a scheme $S$ of
finite type over $k$ at a closed point that carries a family of K3 surfaces
giving rise to our family by completion. Moreover, we may assume that $S$ is
unramified over the functor of polarized K3 surfaces. By Proposition $14$ in
\cite{ogus}, the locus in $S$ parametrizing supersingular K3 surfaces has
dimension $9$. Thus, a general point of $S$ will parametrize a surface of finite
height, and the same will be true of the family restricted to $\Spec A$, as
desired.
\end{proof}

\section{Lifting the Picard group to characteristic $0$}\label{S:lifting}

In this section we prove that any K3 surface $X$ of finite height over a
perfect field $k$ of characteristic $p$ admits a lift to
characteristic $0$ that 
lifts the entire Picard group. While we could use the Nygaard-Ogus
theory of quasi-canonical 
liftings \cite{nygaard-ogus}, this bare lifting fact is much
simpler to prove by analyzing the first-order deformations.

According to Proposition 10.3 of \cite{vandergeer}, if $X$ is a K3
surface of 
finite height over $k$ then the crystalline first Chern class yields an
injective linear map
$$\NS(X)\tensor_\Z k\to\H^1(X,\Omega^1).$$
(Note that the related map
$$\NS(X)\tensor\Z/p\Z\to\H_{\textrm{dR}}^2(X)$$
is always injective, regardless of the height of $X$, by results of Deligne and
Illusie (Remark 3.5 of \cite{deligne}).)

Serre duality yields an isomorphism $\H^1(X,\Omega^1)^\vee\simto\H^1(X,T_X)$.
Given invertible sheaves $L_1,\ldots,L_n$ on $X$, there is a deformation functor
$\Def(X;\{L_1,\ldots,L_n\})$ parametrizing deformations of $X$ together with
deformations of each $L_i$. Given a subset $J\subset\{1,\ldots,n\}$, there is an
associated forgetful natural transformation
$$\Def(X;\{L_i\}_{i=1}^n)\to\Def(X;\{L_j\}_{j\in J}).$$

\begin{prop}\label{finite height snc}
Suppose the isomorphism classes of the $L_j$ generate a subgroup of $\Pic(X)$ of
rank $n$. For each subset $J\subset\{1,\ldots,n\}$, the functor
$\Def(X;\{L_j\}_{j\in J})$ is prorepresentable by a family over
$W\[x_1,\ldots,x_{20-|J|}\]$. Moreover, there is a choice of coordinates for
each subset $J$ such that the forgetful morphism $$\Def(X;\{L_j\}_{j\in
J})\to\Def(X)$$ is identified with the closed immersion associated to the
quotient $$W\[x_1,\ldots,x_{20}\]\to W\[\{x_i\}_{i\not\in J}\].$$
\end{prop}

In particular, if we fix a $\Z$-basis of $\Pic(X)$ we see that the deformation functor $\Def(X;L_1,\ldots,L_n)$ is smooth over $W$.

\begin{cor}
Any K3 surface of finite height over a perfect field $k$ is the closed fiber
of a smooth projective relative K3 surface $\mc X\to\spec W$ in such a way
that the restriction map $\Pic(\mc X)\to\Pic(X)$ is an isomorphism.
\end{cor}
\begin{proof}
Given a basis $L_1,\ldots,L_n$ for $\Pic(X)$, we have by Proposition \ref{finite
height snc} that the deformation functor $\Def(X;L_1,\ldots,L_n)$ is formally
smooth over $W$ of relative dimension $20-n$. Since $W$ is Henselian, the
$k$-valued point $(X;L_1,\ldots,L_n)$ extends to a $W$-valued point, giving a
formal lifting. Since some linear combination of the $L_i$ is ample, we see that
the formal family is formally projective. By the Grothendieck Existence Theorem
this lift is therefore algebraizable as a projective scheme, as desired.
\end{proof}

The proof of Proposition \ref{finite height snc} is an immediate consequence of the following lemma.

\begin{lem}\label{L:cokernel}
Suppose $L_1,\ldots,L_j$ generate a subgroup $\Lambda\subset\NS(X)$ such that
the quotient $\NS(X)/\Lambda$ has trivial $p$-torsion. The cokernel of the
induced morphism $$d\log|_{\Lambda}:\Lambda\tensor k\to\H^1(X,\Omega^1_X)$$ is
naturally identified with the dual of $\Def(X;L_1,\ldots,L_j)(k[\eps])$.
\end{lem}
\begin{proof}
Illusie proved (Proposition IV.3.1.8 of \cite{illusie}) that cupping the Atiyah
class of an invertible sheaf $L$ with the class of a first-order deformation
$[X']\in\H^1(X,T_X)=\Ext^1_X(\Omega^1_X,\ms O_X)$ gives the obstruction class in
$\H^2(X,\ms O)$ to deforming $L$ to $X'$. By Serre duality, the cup product
pairing $\H^1(X,\Omega^1_X)\times\H^1(X,T_X)\to\H^2(X,\ms O)$ is perfect. Since
$d\log$ computes the Atiyah class of $L$, we see that the tangent space
$\Def(X;L_1,\ldots,L_j)(k[\eps])$ is given by the subspace of
$\Hom(\H^1(X,\Omega^1),\H^2(X,\ms O))$
that annihilates $d\log(\Lambda\tensor k)$, as desired.
\end{proof}

\begin{proof}[Proof of Proposition \ref{finite height snc}]
By Deligne \cite{deligne}, each $L_i$ imposes one equation $f_i$ on
$R:=W\[t_1,\ldots,t_{20}\]$. We wish to show that $R/(f_1,\ldots,f_n)$ is
$W$-flat. Consider the Jacobian ideal $J:=(\partial f_i/\partial x_j)$. Reducing
modulo $p$ we see from Lemma \ref{L:cokernel} that $\widebar J\subset R/pR$ is
the unit ideal, whence $J$ is itself the unit ideal. It thus follows that
$R/(f_1,\ldots,f_n)$ is of finite type and formally smooth over $W$, whence it
is smooth, as desired.
\end{proof}

\section{Shioda-supersingular K3s}
\label{S:Shioda}

In this section we assume that $\ch k>2$, so that we can use Ogus's
crystalline theory of supersingular K3 surfaces \cite{ogus-shaf}.  
Fix a K3 surface over $k$ with Picard number $22$. The intersection pairing
makes the Picard group a lattice. Write $N$ for this rank $22$ lattice. Inside
$N\tensor\R$ is the nef cone $\Nef(X)$. It is the closure of the cone generated
by the ample classes in $N$, and thus there is an arithmetic group $O^+$
parametrizing isometries of $N$ whose extension to $\R$ preserves the positive 
cone.

Since $\rk N=22$ and $\dim_k \H^1(X,\Omega^1)=20$, the map $N\tensor
k\to\H_{\text{\rm dR}}^2(X/k)$ has a non-trivial kernel. Ogus has found a canonical
semilinear identification of this kernel with a subspace $K$ of $N\tensor k$. He
proved the following, among other things.

\begin{thm}[Ogus \cite{ogus-shaf}]
An automorphism of $X$ is uniquely determined by its image in $O^+$. Moreover,
the image of $\Aut(X)\to O^+$ is precisely the subgroup of elements that
preserve $K\subset N\tensor k$.
\end{thm}

Let $W$ be the subgroup of $O^+$ generated by reflections in $(-2)$-curves. In
\cite{ogus-shaf}, Ogus proved for supersingular K3s the following analogue of
a well-known classical result.

\begin{prop}\label{P:ogus gp}
The group $W$ is a normal subgroup of $O^+$ that acts on $\Aut(X)$, and the
semidirect product $Aut(X)\ltimes W$ has finite index in $O^+$. Moreover, the
nef cone $\Nef(X)$ is a fundamental domain for the action of $W$ on the positive
cone $\ms C$.
\end{prop}
\begin{proof}
That $W$ is normal in $O^+$ and acts on $\ms C$ with fundamental domain
$\Nef(X)$ is in the series of results proving Proposition 1.10 of
\cite{ogus-shaf}. To see the rest of the Proposition, note that reduction modulo
$p$ defines a map to a finite group $O^+\to\GL(N\tensor\F_p)$. The condition
that an element of the latter group preserve $K$ defines an a priori finite
subgroup, showing that $\Aut(X)$ has finite index in $O^+/W$.
\end{proof}

The following is an immediate consequence of Proposition \ref{P:ogus gp}
\begin{cor}
There is a rational polyhedral fundamental domain for the action of $\Aut(X)$ on
$\Nef(X)$ if and only if there is a rational polyhedral fundamental domain for
the action of $O^+$ on $\ms C$ that lies entirely in $\Nef(X)$.
\end{cor}

These reductions place us in classical territory: the positive cone $\ms C$ is a
standard round cone of the form $x_1>\sqrt{x_2^2+\cdots+x_{22}^2}$. This is one
of the homogeneous self-adjoint convex cones in the sense of \cite{ash-mumford-rap-tai}.

In fact, Sterk shows slightly more than the bare classical statement about the
existence of a rational polyhedral fundamental domain. His proofs carry over to
the present situation verbatim. Write $\Nef(X)_+$ for the convex hull of
$\Nef(X)\cap(N\tensor\Q)$ in $N\tensor\R$.

\begin{prop}[Sterk, Lemma 2.3, Lemma 2.4, Proposition 2.5 of \cite{sterk}] Fix
an ample divisor $H$. The locus of $x\in\Nef(X)_+$ such that $H\cdot
(\phi(x)-x)\geq 0$ for all $\phi\in O^+$ is a rational polyhedral fundamental
domain for the action of $O^+$ on $\ms C$. Moreover, there are finitely many
orbits for the action of $O^+$ on the nodal classes of $X$.
\end{prop}

The reader is referred to \cite{sterk} for the proofs.

\section{The main theorem}\label{S:main}

Let $X$ be a K3 surface over an algebraically closed field $k$ with
characteristic $p\geq 3$. Recall that an element of $\Pic(X)$ is called a
\emph{nodal class\/} if it comes from a smooth rational curve $C\subset X$.

\begin{thm}\label{T:main}
Let $O(\Pic(X))$ be the orthogonal group for $\Pic(X)$ with respect to the
intersection form and let $\Gamma$ be the subgroup of $O(\Pic(X))$ consisting of
elements preserving the nef cone.
\begin{enumerate}
\item The natural map $\Aut(X) \rightarrow \Gamma$ has finite kernel and
cokernel.
\item The group $\Aut(X)$ is finitely generated.
\item The action of $\Aut(X)$ on $\Nef(X)$ has a rational polyhedral fundamental
domain.
\item The set of orbits of $\Aut(X)$ in the nodal classes of $X$ is finite.
\end{enumerate}
\end{thm}
\begin{proof}

When $k$ has characteristic $0$, this is a standard result of Torelli's Theorem
and Sterk's results \cite{sterk}. Now suppose $k$ has positive characteristic
$p$. If we fix a very ample $L$ on $X$, the kernel of the map $\Aut(X)
\rightarrow \Gamma$ is contained in the automorphism group $\Aut(X,L)$ of the
pair (because any such element leaves every invertible sheaf invariant). This latter group is finite type, since it is contained in
$\mathrm{PGL}(N)$ for some $N$, and discrete, since $X$ has no nontrivial vector
fields, thus finite. 

For the remaining claims, we argue as follows.
If $X$ has finite height then we have by Theorem \ref{specialize automorphisms}
and Corollary \ref{C:nef preserved} that $X$ has a lift $X_1$ to an
algebraically closed field of characteristic $0$ such that the specialization
isomorphism $\Pic(X_1)\to\Pic(X)$ is equivariant with respect to the
specialization map $\Aut(X_1)\to\Aut(X)$, and the nef cone specializes to the
nef cone. Since we know all statements for $X_1$, it follows immediately that
the image of $\Aut(X)$ in $\Gamma$ is of finite index, and that $\Gamma$ acts on
$\Nef(X)=\Nef(X_1)$ with a rational polyhedral fundamental domain.
Once we have finiteness of kernel and cokernel, we have that $\Aut(X)$ is
finitely generated because $\Gamma$ is finitely generated (Schreier's
lemma \cite{gps}).

If $X$ is Shioda-supersingular, then the proof is given in Section
\ref{S:Shioda}.
\end{proof}

\begin{remark}
  If we do not assume the Tate conjecture for supersingular K3 surfaces, we can use Section
  \ref{S:violation} to handle case of Artin-supersingular $X$ with Picard number less than 5. By
Proposition 
\ref{P:finitize the height} $X$ is the specialization of a K3
surface $X_1$ 
over an algebraically closed field of characteristic $p$ that has
finite height 
such that the specialization map $\Pic(X_1)\to\Pic(X)$ is an
isomorphism that is 
equivariant with respect to the specialization map
$\Aut(X_1)\to\Aut(X)$. The 
same argument as in the previous paragraph concludes the proof.

\end{remark}

\begin{remark}
  Even in characteristic $0$, the kernel of the map $\Aut(X)\to\Gamma$
  of Theorem \ref{T:main}(i) can be nontrivial. For instance, for a double cover of a plane sextic with Picard rank $1$, the
standard involution acts trivially on $\Pic(X)$.
\end{remark}

Note that a consequence of the argument is that, after choosing a lift to
characteristic $0$ which preserves $\Pic(X)$, the specialization map
$\Aut(X_1)\to\Aut(X)$ has finite index.

Since the finite height argument works for all characteristics, the only
situation not handled by this argument is the case of supersingular K3
surfaces in characteristic $2$. Due to work of Shimada \cite{shimada}, these are
all given as inseparable double covers of $\P^2$; it seems plausible
that a direct analysis along these lines may handle this case as well.


\begin{thebibliography}{10}

\bibitem{artin}
M.~Artin.
\newblock Supersingular K3 surfaces.
\newblock {\em Annales scientifiques de l'\'{E}.N.S. 4th series}, 7(4): 543--567, 1974.

\bibitem{ash-mumford-rap-tai}
A.~Ash, D.~Mumford, M.~Rapoport, Y.~Tai.
\newblock Smooth compactification of locally symmetric varieties.
\newblock {\em Lie Groups: History, Frontiers and Applications,
  Vol. IV}. Math. Sci. Press, Brookline, Mass., iv+335 pp., 1975. 

\bibitem{beauville}
A.~Beauville.
\newblock {\em Complex algebraic surfaces}, second edition.
\newblock Cambridge University Press, 1996.

\bibitem{charles-tate}
  F.~Charles.
  \newblock The Tate conjecture for K3 surfaces over finite
    fields.
   \newblock {\em Invent.\ Math.}, 194(1): 119--145, 2013.

\bibitem{charles-tate-erratum}
  F.~Charles.
  \newblock Erratum to: The Tate conjecture for K3 surfaces over
  finite fields.
  \newblock {\em Invent.\ Math.}, 202(1): 481--485, 2015.

\bibitem{charles-tate-zarhin}
  F.~Charles.
  \newblock Birational boundedness for holomorphic symplectic
  varieties, Zarhin's trick for K3 surfaces, and the Tate conjecture.
  \newblock {\em Ann.\ of Math.\ (2)}, 184(2): 487--526, 2016.

\bibitem{deligne}
P.~Deligne.
\newblock Rel\`{e}vement des surfaces K3 en caract\'{e}ristique nulle.
\newblock {\em Lecture Notes in Mathematics}, 868: 58--79, 1981.

\bibitem{illusie}
L.~Illusie.
\newblock Complexe cotangent et d\'{e}formations. I.
\newblock {Lecture notes in mathematics} 239, 1971.

\bibitem{vandergeer}
G.~van~der~Geer and T.~Katsura.
\newblock On a stratification of the moduli of K3 surfaces.
\newblock {\em Journal of the European Mathematical Society}, 2(3): 259--290, 2000.

\bibitem{lipman}
J.~Lipman.
\newblock Rational singularities with applications to algebraic surfaces and unique factorization.
\newblock {\em Publications math\'{e}matiques de l'I.H.\'{E}.S.}, 36: 195--279, 1969.

\bibitem{mp-tate}
  K.~Madapusi~Pera,
  \newblock The Tate conjecture for K3 surfaces in odd characteristic.
  \newblock {\em Invent.\ Math.}, 201(2):625--668, 2015.

\bibitem{mpw-tate}
  K.~Madapusi~Pera and W.~Kim,
  \newblock 2-adic integral canonical models.
  \newblock {\em Forum Math.\ Sigma} 4: e28, 2016

\bibitem{MM}
T.~Matsusaka and D.~Mumford,
\newblock Two fundamental theorems on deformations of polarized varieties. 
\newblock {\em Amer. J. Math.}, 86: 668--684, 1964.

\bibitem{maulik-tate}
  D.~Maulik,
  \newblock Supersingular K3 surfaces for large primes. With an
  appendix by Andrew Snowden.
  \newblock {\em Duke Math J.} 163(13): 2357--2425, 2014.

\bibitem{nygaard-ogus}
  N.~Nygaard and A.~Ogus.
  \newblock Tate's conjecture for K3 surfaces of finite height
  \newblock {\em  Ann. of Math. (2)} 122(3): 461–507, 1985. 
  
\bibitem{ogus-shaf}
A.~Ogus.
\newblock A crystalline Torelli theorem for supersingular K3 surfaces.
\newblock {\em Progress in Mathematics}, 36: 361--394, 1983.

\bibitem{ogus}
A.~Ogus.
\newblock Singularities in the height strata in the moduli of K3 surfaces.
\newblock {\em Progress in Mathematics}, 195: 325--343, 2001.

\bibitem{shimada}
I.~Shimada.
\newblock Supersingular K3 surfaces in characteristic $2$ as double covers of a projective plane.
\newblock Asian J. Math. 8 (2004), no. 3, 531--586.

\bibitem{gps}
  Á~Seress.
\newblock Permutation group algorithms.
\newblock {\em Cambridge Tracts in Mathematics, 152}. Cambridge
University Press, Cambridge, x+264 pp, 2003.

\bibitem{sterk}
H.~Sterk.
\newblock Finiteness results for algebraic K3 surfaces.
\newblock {\em Mathematische Zeitschrift} 189: 507--513, 1985.

\bibitem{totaro}
B.~Totaro.
\newblock The cone conjecture for Calabi-Yau pairs in dimension 2. 
\newblock {\em Duke Mathematical Journal} 154(2): 241--263, 2010. 

\bibitem{totaro2}
B.~Totaro.
\newblock Algebraic surfaces and hyperbolic geometry. 
\newblock {\em Current developments in algebraic geometry},  MSRI Publications, to appear. 

\end{thebibliography}
\end{document}